\documentclass[11pt,a4paper,reqno]{amsart}%{article}%The reqno document class option displays equation numbers on the right
\usepackage[utf8]{inputenc}
\usepackage{enumitem}
\usepackage{amsmath}
\usepackage{amsfonts}
\usepackage{amssymb}
\usepackage[all]{xy}
\usepackage{xfrac}
\usepackage[yyyymmdd,hhmmss]{datetime}
\usepackage{centernot}
\usepackage[english]{babel}
\usepackage{dsfont}
\usepackage{bbold}
\usepackage{soul}
\usepackage{lmodern}

\theoremstyle{plain}
\newtheorem{theorem}{Theorem}
\newtheorem{proposition}{Proposition}
\newtheorem{lemma}{Lemma}

\theoremstyle{definition}

\author{J. P. Fatelo and N. Martins-Ferreira}

\address{School of Technology and Management, Centre for Rapid and Sustainable Product Development - CDRSP, Polytechnic Institute of Leiria, P-2411-901 Leiria, Portugal.}

\email{jorge.fatelo@ipleiria.pt}

\email{martins.ferreira@ipleiria.pt}

\title[]{A refinement of ternary Boolean algebras}



\thanks{  This work is supported by Fundação para a Ciência e a Tecnologia (FCT\-UID-Multi-04044-2019), Centro2020 (PAMI -- ROTEIRO\-/0328\-/2013\-- 022158) and Polytechnic of Leiria through the projects CENTRO\--01\--0247-FEDER: 069665, 069603, 039958, 039969, 039863, 024533 and also ESTG and CDRSP}

\begin{document}

\begin{abstract}
An algebraic structure with two constants and one ternary operation, which is not completely commutative, is put forward to accommodate ternary Boolean algebras. When the ternary operation is interpreted as Church's conditioned disjunction, Boolean algebras are characterized as a subvariety. Different interpretations for the ternary operation lead to distinct subvarieties. Rings and near-rings of characteristic 2 are used to illustrate the procedure.
\end{abstract}

\maketitle

%%%%%%%%%%%%%%%%%%%%%%%%%%%%%%%%%%%%%%%%%%%%%%%%%%%%%%%%%%%%%%%%%%%%%%
%% MAIN MATTER
%%%%%%%%%%%%%%%%%%%%%%%%%%%%%%%%%%%%%%%%%%%%%%%%%%%%%%%%%%%%%%%%%%%%%%

%\section{Introduction}\label{sec:intro}

%Ternary algebraic operations have been considered now and then for a long time. In particular, 
Ternary Boolean algebras \cite{Grau} were introduced by Grau in 1947 to axiomatize Boolean algebras by means of the ternary operation $xy\circ yz\circ zx$ (see also~\cite{Pad}). In the same year this operation was used independently by Birkhoff and Kiss~\cite{Birk} to characterize distributive lattices. Both approaches are particular cases of median algebras \cite{Bandelt, Isbell, Sholander}.  A different characterization  by Whiteman~\cite{Whiteman} uses ternary rejection $\bar{x}\bar{y}\circ \bar{y}\bar{z}\circ \bar{z}\bar{x}$ instead. Although the set of axioms  is distinct in each case, complete commutativity is a common feature. %A ternary operation is said to be completely commutative when each pair of elements may be interchanged without changing its value. 
In 1948, Church \cite{ChurchPM} shows that it is possible to axiomatize Boolean algebras in terms of the conditioned disjunction $\bar{y}x\circ yz$ which is not completely commutative (see also \cite{hoare}). 
% A Boolean algebra is a lattice which is both distributive and unique complemented.
 As observed in 1949 by Jordan \cite{Jordan}
\emph{the theory of groups would retain relatively little of its charms and richness if it were restricted to the consideration of Abelian groups. It is therefore reasonable to assume that in the theory of lattices as well, dispensing with the commutative law could lead to a considerable broadening of the horizon}.
 A systematic study of non-commutative Boolean algebras and skew lattices has been reinforced by Leech \cite{Lee90,Lee96}. Although skew lattices are not necessarily commutative, there are good reasons why they should be idempotent \cite{Jordan} (see also \cite{Lee89,Lee19}). 
 However, this is too restrictive if we  wish to generalise Boolean algebras from Boolean rings (see \cite{Salibra3} for more recent generalisations).

%Indeed, any ring which is idempotent is necessarily commutative and has characteristic two ($x+x=0$). 
%But this means that if considered as a ringThis means that any generalization of Boolean algebras retaining idempotency is inconsistent with the ring structure being non-commutative. So, in order to consider non-commutative Boolean rings we need to disregard idempotency and apparently move away from skew lattices. Nevertheless, as we will see, there is a possibility to reunify both views if negation is maintained as part of the structure.

% Motivated by ternary affine geometry \cite{affine,Bertram}, the study of geodesics from an algebraic point of view \cite{mobi,Mobi2sphere} together with the observation that Church's conditioned disjunction $\bar{y}x\circ yz$ is similar to the affine operation $\bar{y}x+yz$ has suggested to us that the set of axioms \ref{C1}--\ref{C4} considered here is appropriate for the study of non-commutative and non-idempotent Boolean algebras.

In this note, we propose a system $(A,p,0,1)$ with two constants $0$ and~$1$ and a ternary operation $p$ satisfying conditions \ref{C1}--\ref{C4}, presented in Lemma~\ref{lemma:1}, from which
an involution $\bar{x}=p(1,x,0)$ and associative binary operations $x\cdot y=p(0,x,y)$, $x\circ y=p(x,y,1)$ and $x+y=p(x,y,\bar{x})$ are derived  (see also \cite{hoare}). Every formula specifying $p$ in terms of its derived operations is equivalent to adding a new axiom and hence it gives rise to a subvariety.
For example, Boolean algebras form a subvariety if and only if $p(x,y,z)$ is interpreted as Church's conditioned disjunction $\bar{y}x\circ yz$ 
(Theorem~\ref{thm:1}). 
%In a Boolean algebra $(A,\circ,\cdot,\bar{()},0,1)$, or in a Boolean ring $(A,+,\cdot,0,1)$, the formulas $\bar{y}x\circ yz$ and $\bar y x+yz$ are equal. Nevertheless, 
Moreover, if specifying $p(x,y,z)$ as $\bar y x+yz$ or $(1+y)x+yz$ then the subvariety of unitary rings with characteristic~2 is obtained (Theorem~\ref{thm:2}). While if \mbox{$p(x,y,z)$ is $x+y(x+z)$} then the result is the subvariety of unitary Abelian (right) near-rings with characteristic $2$ (Theorem \ref{thm:3}). 
Note that the previous formulas for~$p$ are all equal in the context of Boolean algebras or Boolean rings and are related to the affine formula %\cite{Bertram, mobi, affine}.
$x+y(z-x)$ used in Proposition \ref{prop}.
%More generally, the affine formula \cite{Bertram, mobi, affine} i.e. $x+y(z-x)$ may be used when the characteristic is not necessarily $2$ (Proposition \ref{prop}).

%This shows that small variations on $x+y(z-x)$, which may be interpreted as an affine formula \cite{Bertram, mobi, affine}, give rise to different subvarieties (see also Proposition \ref{prop}).

Lemma \ref{lemma:1} presents the algebraic structure $(A,p,0,1)$ satisfying axioms \ref{C1}--\ref{C4} and gives the properties needed for the upcoming theorems. Axiom~\ref{C4} alone defines a Church algebra~\cite{Salibra2,Salibra1} while a Menger algebra of rank 2 (see e.g. \cite{menger}) uses axiom \ref{C3}.
The axioms \ref{C1}, \ref{C3} and \ref{C4} have been used to define  proposition algebras \cite{Ponse} while a generalisation of axioms \ref{C2}--\ref{C4} has been used to study spaces with geodesic paths~\cite{mobi,affine,mobi2sphere}.
%%%%%%%%%%%%%%%%%%%%%%%%%%%%%%%%%%%%%%%%%%%%%%%%%%%%%%%%%%%%%%
\begin{lemma}\label{lemma:1}
Let $(A,p,0,1)$ be a system consisting of a set $A$, together with a ternary operation $p$ and two constants $0,1\in A$ satisfying:
\begin{enumerate}[label={\bf (T\arabic*)}]
\item\label{C1}\label{smA2} $p(0,a,1)=a$%=p(a,0,b)=p(b,1,a)$
\item\label{C2}\label{smA3} $p(a,b,a)=a$
\item\label{C3}\label{smA7} $p(a,p(b_1,b_2,b_3),c)=
p(p(a,b_1,c),b_2,p(a,b_3,c))$
\item\label{C4}\label{smA4}\label{smA5} $p(a,0,b)=a=p(b,1,a)$.
\end{enumerate}
For $\bar{a}=p(1,a,0)$, $ a\cdot b=p(0,a,b)$, $ a\circ b=p(a,b,1)$ and $a+b=p(a,b,\bar a)$, the following properties hold:
\begin{enumerate}[label={\rm(L\arabic*)}]
\item $\bar{1}=0,\quad \bar{0}=1$ \quad and \quad $\overline{\overline{a}}=a$
\item $\overline{p(a,b,c)}=p(\overline{a},b,\overline{c}), \quad p(a,b,c)=p(c,\bar{b},a)$
\item\label{P0} $\overline{a\cdot b}=\overline{b}\circ \overline{a},\quad\overline{a\circ b}=\overline{b}\cdot \overline{a},\quad 
\overline{a+b}=\bar a+b=a+\bar b$
\item\label{monoids} $(A,\cdot,1)$, $(A,\circ,0)$ and $(A,+,0)$ are monoids
\item\label{absorb} $a\cdot 0=0 =0\cdot a,\quad a\circ 1=1=1\circ a$
\item\label{P4-0} $a+1=\bar a=1+a,\quad 1+1=0$
%\item $p(a,b,c)\cdot d=p(a\cdot d,b,c\cdot d),\quad a\circ p(b,c,d)=p(a\circ b,c,a\circ d)$, \newline $a+p(b,c,d)=p(a+b,c,a+d)$.
%\item $p(a,b,p(a,d,c))=p(a,b\cdot d,c).$
\end{enumerate}
Furthermore:
\begin{enumerate}[label={\rm(L\arabic*)},resume]
\item\label{P3} if, for all $a,b\in A$, $p(0,a,b)=p(a,a,b)$ then the systems $(A,\cdot,\bar{()},0)$ and $(A,\circ,\bar{()},1)$ are commutative magmas with $a\cdot \bar{a}=0$ and $a\circ \bar{a}=1$;
\item\label{P1} if $\bar{a}\cdot a=0$ then $a\cdot a=a$; if $\bar{a}\circ a=1$ then $a\circ a=a$;
\item\label{P2} if either $(A,\cdot,1)$ or $(A,\circ,0)$ is a commutative and idempotent  monoid then $(A,\circ,\cdot,0,1)$ is a bounded distributive lattice;
\item\label{P4} if $a+a=0$ then $+$ is commutative and $(a+b)\cdot c=a\cdot c+b\cdot c$.
\end{enumerate}
\end{lemma}
\begin{proof}
The results are readily obtained from the axioms. Explicit proofs can be found in \cite{preprint-TBA}.
\end{proof}

The following theorem is a refinement of Grau's  ternary Boolean algebra in the sense that it uses Church's  operation and a systematisation of Hoare's axioms considered in \cite{hoare}.

\begin{theorem}\label{thm:1}
Suppose that $(A,p,0,1)$ satisfies axioms \ref{C1} to \ref{C4}. For $\bar{a}=p(1,a,0)$, $ a\cdot b=p(0,a,b)$, $ a\circ b=p(a,b,1)$ and $a+b=p(a,b,\bar a)$, the following conditions are equivalent:
\begin{enumerate}[label={\rm(\roman*)}]
\item\label{T1} $(A,\circ,\cdot,\bar{()},0,1)$ is a Boolean algebra
\item\label{T0} $(A,+,\cdot,0,1)$ is a Boolean ring
\item\label{T2} $p(a,b,c)=(\bar{b}\cdot a)\circ (b\cdot c)$
\item\label{T3} $p(a,a,b)=a\cdot b$
\item\label{T4} $p(a,b,b)=a\circ b$.
\end{enumerate}
\end{theorem}
\begin{proof}%{(of Theorem \ref{thm:1})}

\noindent Conditions \ref{T3} and \ref{T4} are equivalent by duality~\ref{P0}. Condition~\ref{T3} implies the commutativity of $\cdot$ and $\circ$  as well as the fact that $\bar{()}$ is a Boolean complement \ref{P3}, which in turn implies idempotency \ref{P1} and distributivity~\ref{P2}. Associativity is ensured by \ref{monoids} and hence \ref{T3} implies~\ref{T1}.
Condition~\ref{T2} follows from \ref{T1} as the unique solution $p=p(a,b,c)$ to the system of equations
$$\left\{\begin{array}{rcl}
p\cdot c&=&((\bar{b}\cdot a)\circ(b\cdot c))\cdot c\\
c\circ p&=&c\circ ((\bar{b}\cdot a)\circ(b\cdot c))
\end{array}\right..$$
Condition \ref{T3} follows from \ref{T2}. Indeed, \ref{T2} implies $1=p(1,a,1)=\bar{a}\circ a$ and (by duality) also $\bar{a}\cdot a=0$. Moreover,
$p(a,a,b)=(\bar{a}\cdot a)\circ (a\cdot b)=a\cdot b$. It is readily checked that \ref{T1} and \ref{T0} are equivalent \cite{preprint-TBA}.
\end{proof}
Note that the ternary Mal'tsev operation~\cite{Maltsev} in a group $(ab^{-1}c)$ admits a similar characterisation (see Theorem 4 in \cite{Certaine}). However, as pointed out by Birkhoff and Kiss \cite{Birk}, the two operations are quite different: in a group we have $p(a,b,b)=a=p(b,b,a)$ whereas in a Boolean algebra we have $p(a,b,a)=a$. Nevertheless, as shown in~\cite{NMF2012} there are some touching points between (weakly) Mal’tsev categories and distributive lattices. 

Each new interpretation of the ternary operation $p$ satisfying \mbox{\ref{C1}--\ref{C4}} in terms of its derived operations is equivalent to adding a new axiom and gives rise to a new subvariety, as Theorem \ref{thm:2} and Theorem \ref{thm:3} illustrate. An example of a ternary operation obtained from a unitary Abelian near-ring~\cite{near-rings-2021}, which is not necessarily determined by its derived operations, is presented in the next proposition. The algebraic model of the unit interval considered in \mbox{\cite{ccm_magmas, mobi}} is another example.
\begin{proposition}\label{prop}
If $(A,+,\cdot,0,1)$ is a unitary Abelian (right) near-ring, in which $a\cdot 0=0$, then $(A,p,0,1)$ with $p(a,b,c)=a+b(c-a)$ satisfies the axioms \ref{C1} to \ref{C4}.
\end{proposition}
\begin{proof} The proof is straightforward.\end{proof}

In Theorem \ref{thm:2} below, the formula for $p$ is obtained by writing $a+b(c-a)$ as $(1-b)a+bc$ and replacing $1-b$ by $1+b$. Recall that a ring (or a near-ring) of characteristic 2 is such that $b+b=0$ for all $b$, so that $1-b=1+b$. Moreover, rewriting $(1+b)a+bc$ as $a+b(a+c)$ gives the formula used in Theorem \ref{thm:3}.
\begin{theorem}\label{thm:2}
Suppose that $(A,p,0,1)$ satisfies axioms \ref{C1} to \ref{C4}. For $\bar{a}=p(1,a,0)$, $ a\cdot b=p(0,a,b)$ and $a+b=p(a,b,\bar a)$, the following conditions are equivalent:
\begin{enumerate}[label={\rm(\roman*)}]
\item\label{T1a} $(A,+,\cdot,0,1)$ is a unitary ring of characteristic $2$
%\item\label{T2a} $p(a,b,c)=a+(b\cdot a) + (b\cdot c)$
\item\label{T2a} $p(a,b,c)=(\bar b\cdot a) + (b\cdot c)$
\item\label{T3a} $a\cdot(b+c)=(a\cdot b)+(a\cdot c)$.
\end{enumerate}
\end{theorem}
\begin{proof}
It is clear that \ref{T1a} implies \ref{T3a}. By \ref{P4-0}, \ref{T3a} implies $a+a=0$ and hence, considering \ref{monoids} and \ref{P4}, \ref{T3a} implies~\ref{T1a}. Moreover, when~$a+a=0$:
\begin{equation}\label{magic}
a+p(a,b,c)=p(a,p(a,b,c),\bar a)\ =\ p(p(a,a,\bar a),b,p(a,c,\bar a))=b\cdot(a+c).
\end{equation}
Consequently, \ref{T3a} implies \ref{T2a} as $p(a,b,c)=a+ba+bc=(1+b)a+bc.$
It remains to prove that \ref{T2a} implies~\ref{T3a}. By \ref{C2}, \ref{monoids} and \ref{P4-0}, \ref{T2a} implies $1=p(1,a,1)=\bar a+a=1+a+a$ i.e. $a+a=0$. Then (\ref{magic}) and \ref{P4} imply left distributivity:
$
a\cdot(b+c)=b+p(b,a,c)=b+(1+a)b+ac=ab+ac.
$
\end{proof}
\noindent The unique non-commutative ring of order $8$, say consisting of all upper triangular binary  $2$-by-$2$ matrices, illustrates Theorem \ref{thm:2}. Note that addition is the Boolean symmetric difference as in a Boolean ring.

\begin{theorem}\label{thm:3}
Suppose that $(A,p,0,1)$ satisfies axioms \ref{C1} to \ref{C4}. For $\bar{a}=p(1,a,0)$, $ a\cdot b=p(0,a,b)$ and $a+b=p(a,b,\bar a)$, the following conditions are equivalent:
\begin{enumerate}[label={\rm(\roman*)}]
\item\label{T1b} $(A,+,\cdot,0,1)$ is a unitary (right) near ring of characteristic~$2$
\item\label{T2b} $p(a,b,c)=a+(b\cdot(a+c))$
\item\label{T3b} $a+a=0$
\item\label{T4b} $(a+b)\cdot c=(a\cdot c)+(b\cdot c)$.
\end{enumerate}
\end{theorem}
\begin{proof}
It is clear that \ref{T1b} implies \ref{T3b} and, considering \ref{monoids} and \ref{P4}, \ref{T3b} implies~\ref{T1b}. Using (\ref{magic}), condition \ref{T3b} implies \ref{T2b}:
\[
a+p(a,b,c)=b\cdot(a+c)=a+(a+b\cdot(a+c)).
\]
Using \ref {monoids} and \ref{C4}, condition \ref{T2b} implies \ref{T3b}: $0=p(a,1,0)=a+a.$
Properties \ref{P4-0} and \ref{P4} show that \ref{T3b} and \ref{T4b} are equivalent.
\end{proof}
\noindent Several examples of unitary (right) near-rings of characteristic 2 with four elements can be found. The following example illustrates Theorem \ref{thm:3}. Multiplication is neither commutative nor idempotent and, once again, addition is the same as Boolean symmetric difference. Note that if the formula $(1+y)x+yz$ is used as the ternary operation $p$ instead of $x+y(x+z)$ then \ref{C3} is not satisfied.
\[
\begin{array}{c|cccc}
%\hline 
\cdot & 0 & u & v & 1 \\ 
\hline 
0  & 0 & 0 & 0 & 0 \\ 
%\hline 
u & 0 & 0 & 0 & u \\ 
%\hline 
v & 0 & u & v & v \\ 
%\hline 
1 & 0 & u & v & 1 \\ 
%\hline 
\end{array} \qquad\quad
\begin{array}{c|cccc}
%\hline 
+ & 0 & u & v & 1 \\ 
\hline 
0  & 0 & u & v & 1 \\ 
%\hline 
u  & u & 0 & 1 & v \\ 
%\hline 
v  & v & 1 & 0 & u \\ 
%\hline 
1  & 1 & v & u & 0 
%\hline 
\end{array}
\] 

The results presented in this note suggest that unitary Abelian near-rings of characteristic~$2$, considered as generalised Boolean rings, provide a way to explore skew lattices that are not necessarily idempotent.
%not  axiomatized in terms of the ternary operation $\bar{y}x+yz$ which is not completely commutative. This is possible using a set of axioms that goes beyond Boolean algebras in a way that keeps pseudo-complements but does not require commutativity nor distributivity. 
This line of research goes in the direction advocated by Birkoff and Von Neumann~\cite{Birk-Neumann} as a possibility for the logic of quantum mechanics.

\end{document}